\documentclass[amsart,10pt,oneside,openright]{article}
\usepackage[english]{babel}
\usepackage[colorlinks,citecolor=green,linkcolor=red]{hyperref}
\usepackage{amsthm}
\usepackage{color}
\usepackage{mathrsfs}
\usepackage{amsmath}
\usepackage{mathtools}
\usepackage{amsfonts}
\usepackage{amssymb}
\usepackage{bm}
\usepackage{physics}
\usepackage{enumitem}
\usepackage{a4wide}
\usepackage{cleveref}
\usepackage{esint}
\usepackage{nicefrac}
\usepackage{yfonts}
\usepackage{dirtytalk}
\usepackage{autonum}
\usepackage{a4wide}

\newtheorem{thm}{Theorem}
\newtheorem*{thm*}{Theorem}
\newtheorem{lem}[thm]{Lemma}
\newtheorem{prop}[thm]{Proposition}

\theoremstyle{definition}
\newtheorem{defn}[thm]{Definition}
\newtheorem{ass}{Assumption}

\theoremstyle{remark}
\newtheorem{rem}[thm]{Remark}

\newcommand{\fr}{\penalty-20\null\hfill$\blacksquare$}

\newcommand{\heis}{{\mathbb{H}}} 	
\newcommand{\defeq}{\mathrel{\mathop:}=}

\def\Xint#1{\mathchoice
	{\XXint\displaystyle\textstyle{#1}}%
	{\XXint\textstyle\scriptstyle{#1}}%
	{\XXint\scriptstyle\scriptscriptstyle{#1}}%
	{\XXint\scriptscriptstyle\scriptscriptstyle{#1}}%
	\!\int}
\def\XXint#1#2#3{{\setbox0=\hbox{$#1{#2#3}{\int}$}
		\vcenter{\hbox{$#2#3$}}\kern-.5\wd0}}
\def\dashint{\Xint-}

\newcommand{\RCD}{{\mathrm {RCD}}}

\newcommand{\PI}{{\mathrm {PI}}}

\newcommand{\XX}{{\mathsf{X}}}
\newcommand{\YY}{{\mathsf{Y}}}
\newcommand{\ZZ}{{\mathsf{Z}}}
\newcommand{\dist}{{\mathsf{d}}}
\renewcommand{\dd}{\dist}
\newcommand{\mass}{{\mathsf{m}}}

\newcommand{\Ch}{{\mathrm{Ch}}}

\newcommand{\LIP}{{\mathrm {LIP}}}

\newcommand{\lip}{{\mathrm {lip}}}

\newcommand{\DIFF}{{\mathrm{D}}}

\newcommand{\RR}{\mathbb{R}}
\newcommand{\NN}{\mathbb{N}}
\newcommand{\LL}{\mathcal{L}}
\newcommand{\HH}{\mathcal{H}}

\let\oldchi=\chi
\renewcommand{\chi}{\text{\raisebox{\depth}{\(\oldchi\)}}}

\let\phi\varphi
\let\epsilon\varepsilon
\title{Nguyen's approach to Sobolev spaces\\ in metric measure spaces \\with unique tangents}

\author{Camillo Brena\footnote{\href{mailto:camillo.brena@sns.it}{camillo.brena@sns.it}, Scuola Normale Superiore, Piazza dei Cavalieri, 7, 56126 Pisa, Italy.}
\and Andrea Pinamonti\footnote{\href{mailto:andrea.pinamonti@unitn.it}{andrea.pinamonti@unitn.it}, Dipartimento di Matematica, Universit\`{a} degli Studi di Trento, Via Sommarive 14, 38123 Povo (Trento), Italy.}}

\begin{document}
	
 \maketitle
\begin{abstract}
    We extend Nguyen's characterization of Sobolev spaces $W^{1,p}$ \cite{Ngu06,Ngu08} to the setting of PI-metric measure spaces such that at $\mass$-a.e.\ point the tangent space (in the Gromov--Hausdorff sense) is unique and Euclidean with a fixed dimension. 
    We also generalize \cite{CLL14} to PI-metric measure spaces such that at $\mass$-a.e.\ point the tangent space is unique and equal to the Heisenberg group with a fixed homogeneous dimension. 
    The approach is easier and  completely different from the one in \cite{Ngu06,Ngu08}.
\end{abstract}
 \section{Introduction}

In the early 2000’s the study of fractional $s$-seminorms gained great interest when Bourgain, Brezis and Mironescu \cite{Bourg} on one hand, and Maz’ya and Shaposhnikova \cite{MazS} on the other, showed that they can be seen as functionals interpolating between the $L^p(\RR^n)$-norm and the $W^{1, p}(\RR^n)$-seminorm. More precisely, the following  important asymptotic formulas have been proved in \cite{Bourg} and \cite{MazS} respectively: for any $p\in (1,\infty)$, $s\in (0, 1)$, $n\in \NN$ and $f\in \bigcup_{0<s<1} W^{s, p}(\RR^n) \cap W^{1, p}(\RR^n)$ with bounded support, it holds
 \begin{equation}\label{BBM:intro}
 \mathop{\lim}_{s \nearrow 1}~ (1-s)\|f\|^p_{W^{s, p}(\RR^n)}= {K }  \| \DIFF f\|_{L^p(\RR^n)}^p  \tag{BBM}
\end{equation}
and
 \begin{equation}\label{MS:intro}
 \mathop{\lim}_{s \searrow 0} ~s\|f\|^p_{W^{s, p}(\RR^n)}= {L }  \|  f\|_{L^p(\RR^n)}^p, \tag{MS}
\end{equation}
where $K=K(p,n)>0$, $L=L(p,n)>0$ and $W^{s, p}(\RR^n)$ is defined as the set of $L^p(\RR^n)$ functions with finite seminorm
\[
\|f\|_{W^{s, p}(\RR^n)}\defeq\left (\int_{\RR^n} \int_{\RR^n} \frac{|f(x)-f(y)|^p} {|x-y|^{n+sp}}\dd x \dd y \right)^{\frac 1p}.
\]
Later on, several different generalizations of these results in Euclidean spaces were considered: see e.g.\ Ponce \cite{Ponce},
Leoni--Spector \cite{LS1,LS2}, Brezis--Nguyen \cite{BNgu,BNgu2,BNgu3,BNgu4}, Pinamonti--Vecchi--Squassina \cite{PVS1},
Nguyen--Pinamonti--Vecchi--Squassina \cite{HPV}, Brezis--Van Schaftingen--Yung \cite{BVY,BVY2,BVY3}, Maalaoui--Pinamonti \cite{MalPin},
Garofalo--Tralli \cite{GarTral2},  Alonso Ruiz--Baudoin--Chen--Rogers--Shanmugalingam--Teplyaev \cite{Baud1}, Alonso Ruiz--Baudoin \cite{Baud}.\\

Motivated by the previous results in \cite{Ngu06,Ngu08,Ngu11,Ngu11A}, H.M.\ Nguyen 
provided a new characterization of Sobolev spaces using a completely different family of nonlocal functionals, see also \cite{BouNgu}. Namely,
\begin{thm}\label{thmlimitsEuclideo}
		Let $p\in(1,\infty)$ and $f\in W^{1,p}(\RR^n)$.  Then
		\begin{equation}\label{mainthmeqeuclideo}
			\lim_{\delta\searrow 0} \iint_{|f(x)-f(y)|>\delta}\frac{\delta^p}{|x-y|^{n+p}}\dd y\dd x=\frac{C_{n,p}} {p}\int_{\RR^n}|\DIFF f|^p(x)\dd x
		\end{equation}
	and 
		\begin{equation}\label{mainthmeqeuclideo1}
		\lim_{\delta\searrow 0}\iint_{|f(x)-f(y)|\le 1}\frac{\delta|f(x)-f(y)|^{p+\delta}}{|x-y|^{n+p}}\dd y\dd x=C_{n,p}\int_{\RR^n}|\DIFF f|^p(x)\dd x,
	\end{equation}
		where $$C_{n,p}\defeq \int_{v\in S^{n-1}}|e\,\cdot\,v|^p\dd\HH^{n-1}(v)\qquad\text{for any }e\in S^{n-1}.$$
\end{thm}
Similar results hold for BV spaces \cite{Davila,Ponce}, for magnetic Sobolev spaces \cite{HPV,HS19} and for the Heisenberg group \cite{CLL14}.\\
In \cite[Remark 6]{Brez02}, Brezis suggested generalizing the theory to more general metric measure spaces $(\XX, \dd, \mass)$. One of the first attempts in this direction was given by Di Marino--Squassina \cite{DiMarSquas}, who assumed the measure $\mass$ to be doubling and the space to support a $(1, p)$-Poincar\'e inequality. They showed, among other results, the following
\begin{thm}\label{cart}
			Let $(\XX,\dist,\mass)$ be a $p$-$\PI$ space for some $p\in(1,\infty)$. 
Then, there exist two constants $C_L,C_U> 0$, such that for every $f\in L^p(\XX)$, the following holds:
\begin{equation}
C_L{\rm Ch}_p(f)\leq \limsup_{\rho\searrow 0}\iint_{|f(x)-f(y)|>\delta}\frac{\delta^p}{\mass(B_{\dd(x,y)}(x))\dd(x,y)^p}\dd\mass(x)\dd\mass(y)\leq C_U{\rm Ch}_p(f),
\end{equation} 
where ${\rm Ch}_p(f)$ is defined in \eqref{Cheeg}.
\end{thm}
Moreover, a closer look at \cite{DiMarSquas} provides the following two theorems. In particular, Theorem \ref{mainthm} is an immediate consequence of Theorem \ref{cart}, which is \cite[Theorem 1.5]{DiMarSquas}, whereas Theorem \ref{mainthm1} is extracted from the proof of  \cite[Theorem 1.5]{DiMarSquas}. As the second result is not explicitly stated in \cite{DiMarSquas}, we add a few lines concerning its proof in Section \ref{sectproof}.
	\begin{thm}\label{mainthm}
		Let $(\XX,\dist,\mass)$ be a $p$-$\PI$ space for some $p\in(1,\infty)$.
		Let $f\in L^{p}(\XX)$. Then $f\in W^{1,p}(\XX)$ if and only if 
		$$
		\sup_{\delta\in(0,1)}\iint_{|f(x)-f(y)|>\delta}\frac{\delta^p}{\mass(B_{\dist(x,y)}(x))\dist(x,y)^p}\dd\mass(y)\dd\mass(x)<\infty.
		$$		
	\end{thm}

\begin{thm}\label{mainthm1}
	Let $(\XX,\dist,\mass)$ be a $p$-$\PI$ space for some $p\in(1,\infty)$.
	Let $f\in L^{p}(\XX)$. Then $f\in W^{1,p}(\XX)$ if and only if 
	$$
		\sup_{\delta\in(0,1)}\iint_{|f(x)-f(y)|\le 1}\frac{\delta|f(x)-f(y)|^{p+\delta}}{\mass(B_{\dist(x,y)}(x))\dist(x,y)^p}\dd\mass(y)\dd\mass(x)<\infty
	$$
	and 
	$$
	\iint_{|f(x)-f(y)|>  1}\frac{1}{\mass(B_{\dist(x,y)}(x))\dist(x,y)^p}\dd\mass(y)\dd\mass(x)<\infty.$$
\end{thm}
In this paper we generalize Theorem \ref{thmlimitsEuclideo}  to the context of $\PI$ metric measure spaces admitting unique tangent cone. To be more precise, we consider $\PI$  metric measure spaces such that their tangents (in the Gromov--Hausdorff sense) for $\mass$-a.e.\ $x\in \XX$ are either Euclidean spaces with a fixed dimension $n$ or the Heisenberg group with a fixed homogeneous dimension $Q$. This class contains many relevant examples, for instance closed Riemannian manifolds, weighted Euclidean spaces for continuous weights bounded from above and below, $\RCD(0,N)$ spaces, Ahlfors $\RCD(K,N)$ spaces, $\RCD(K,N)$ spaces  with  finite measure and the Heisenberg group.
Let us explicitly point out that the proof of Theorem \ref{mainthmeqeuclideo} deeply relies on some important geometric facts, such as that any two points can be connected by a line-segment and the uniformity in every direction of the unit sphere. Both those facts are precluded in general metric measure spaces. For this reason our proof is completely different from the classical one and it resembles the one, via a blow-up, used in \cite{Gorny} to prove a Bourgain--Brezis--Mironescu type theorem in the metric setting, see \cite{BPP,LaPiZh22}. 
\subsection{Statement of results}
Let us now precisely state the main results of this paper. We start setting the main assumptions on the space $(\XX,\dist,\mass)$ (see \eqref{normalized1} and \eqref{normalized2} for the definition of the renormalized measures):
	\begin{ass}\label{Ass1}
	We say that a $\PI$ space $(\XX,\dist,\mass)$ satisfies Assumption \ref{Ass1}  if \begin{equation} \label{tangA}
	{\mathrm {Tan}}_x(\XX,\dist,\mass)=\{(\RR^n,\dist_e,\underline{\LL}^n,0)\}\qquad\text{for $\mass$-a.e.\ $x\in\XX$}\end{equation}
	and \begin{equation}\label{densityA}
		\exists\Theta_n(\mass,x)\defeq\lim_{r\searrow 0}\frac{\mass(B_r(x))}{\omega_n r^n}\in (0,\infty)\qquad\text{for $\mass$-a.e.\ $x\in\XX$}.
	\end{equation}
\end{ass}
Alternatively,
\begin{ass}\label{Ass2}
	We say that a $\PI$ space $(\XX,\dist,\mass)$ satisfies Assumption \ref{Ass2}  if 
 \begin{equation}\label{cdsasacds}
     {\mathrm {Tan}}_x(\XX,\dist,\mass)=\{(\mathbb{H}^n,\dist_{\mathbb{H}},\underline{\LL}^{2n+1},0)\}\qquad\text{for $\mass$-a.e.\ $x\in\XX$}
 \end{equation}
	and \begin{equation}\label{density2}
		\exists\Theta_n(\mass,x)\defeq\lim_{r\searrow 0}\frac{\mass(B_r(x))}{\omega_n r^{2n+2}}\in (0,\infty)\qquad\text{for $\mass$-a.e.\ $x\in\XX$}.
	\end{equation}
\end{ass}
We point out that spaces satisfying Assumption \ref{Ass1} are $n$-rectifiable (\cite{BateMM}). A similar conclusion holds for spaces satisfying Assumption \ref{Ass2}, provided that \cite[Conjeture 7.1]{Ant23} holds true.

\begin{defn}\label{defnormconst}
    We define, for $n\in\NN$ and $p\in(1,\infty)$,
\begin{equation}
	C_{n,p}\defeq\begin{cases}
	    \displaystyle\int_{\{v\in \RR^n:|v|=1\}}|e\,\cdot\,v|^p\dd\HH^{n-1}(v)\qquad&\parbox{15em}{for any $e\in \RR^{n}$ with $|v|=1$,\\in the case of Assumption \ref{Ass1},}\\\\\displaystyle\int_{\{v\in \mathbb{H}^n:\|v\|=1\}}|\zeta\,\cdot\, v|^p\dd \sigma(v)\qquad&\parbox{20em}{for any horizontal $\zeta\in \mathbb{H}^n$ with $\|\zeta\|=1$,\\in the case of Assumption \ref{Ass2},}
	\end{cases} 
\end{equation}
where $\sigma$ is a suitable Radon measure, see the proof of Lemma \ref{blowuph}.
\end{defn}

We state now the main result of this note. Its proof is deferred to Section \ref{sectproof}.
\begin{thm}\label{thmlimits}
		Let $(\XX,\dist,\mass)$ be a $p$-$\PI$ space for some $p\in(1,\infty)$, satisfying either Assumption \ref{Ass1} or Assumption \ref{Ass2}. Let   $f\in W^{1,p}(\XX)$.  Then
		\begin{equation}\label{mainthmeq}
			\lim_{\delta\searrow 0} \iint_{|f(x)-f(y)|>\delta}\frac{\delta^p}{\mass(B_{\dist(x,y)}(x))\dist(x,y)^p}\dd\mass(y)\dd\mass(x)=\frac{C_{n,p}} {p \omega_n}\int_\XX|\DIFF_p f|^p(x)\dd\mass(x)
		\end{equation}
	and 
		\begin{equation}\label{mainthmeq1}
		\lim_{\delta\searrow 0}\iint_{|f(x)-f(y)|\le 1}\frac{\delta|f(x)-f(y)|^{p+\delta}}{\mass(B_{\dist(x,y)}(x))\dist(x,y)^p}\dd\mass(y)\dd\mass(x)=\frac{C_{n,p}} { \omega_n}\int_\XX|\DIFF_p f|^p(x)\dd\mass(x),
	\end{equation}
		where $C_{n,p}$ is as in Definition \ref{defnormconst}.
\end{thm}
As far as we know, Theorem \ref{thmlimits} represents the first extension of Theorem \ref{thmlimitsEuclideo} to a large class of metric measure spaces and it paves the way to the study of more general functionals in the setting of metric measure spaces. We remark that the assumption that the tangent space is in a certain sense unique is rather sharp, see e.g.\ \cite[Example 4.4]{Gorny}.

\section{Preliminaries}
In this section we collect some preliminary notions that we will use throughout the paper. We use standard notation for the functional spaces (e.g.\ $\LIP$ denotes the space of Lipschitz functions) and we add the subscript $\mathrm{bs}$ to denote the subspace of functions with bounded support.
\\

\noindent\textbf{Sobolev spaces.}\\
By metric measure space we mean a triplet \((\XX,\dist,\mass)\), where
$(\XX,\dist)$ is a complete and separable metric space and $\mass$ is a Borel measure on $\XX$ that is finite on balls. A pointed metric measure is a quadruplet \((\XX,\dist,\mass,x)\), where \((\XX,\dist,\mass)\) is a metric measure space and $x\in\XX$.

Given any \(p\in(1,\infty)\), the Cheeger \(p\)-energy \({\rm Ch}_p(f)\) of a function \(f\in L^p(\mass)\) is defined as
\begin{equation}\label{Cheeg}
{\rm Ch}_p(f)\coloneqq\inf_{(f_n)}\liminf_{n}\int\lip(f_n)^p\dd\mass,
\end{equation}
where the infimum is taken among all sequences  \((f_n)_n\subseteq \LIP_{\mathrm{bs}}(\XX)\cap L^p(\mass)\)
such that \(f_n\to f\) in \(L^p(\mass)\). Here, the slope \(\lip(g)\colon\XX\to[0,+\infty)\) of $g\in\LIP(\XX)$ is
\[
\lip(g)(x)\defeq\limsup_{y\to x}\frac{|g(x)-g(y)|}{\dist(x,y)},
\]
where we understand \(\lip(g)(x)= 0\) if $x$ is isolated. The Sobolev space \(W^{1,p}(\XX)\) is defined as
\[
W^{1,p}(\XX)\coloneqq\big\{f\in L^p(\mass)\;\big|\;{\rm Ch}_p(f)<+\infty\big\}.
\]
Given any \(f\in W^{1,p}(\XX)\), there exists a unique function \(|\DIFF_p f|\in L^p(\mass)\), called the {minimal \(p\)-relaxed slope}
of \(f\), and the Cheeger \(p\)-energy of \(f\) can be represented as
\[
{\rm Ch}_p(f)=\int|\DIFF_p f|^p\dd\mass.
\]
See \cite{AmbrosioGigliSavare11,AmbrosioGigliSAvare11-3} after \cite{Cheeger00} for more details about the above notions and results. We stress that $|\DIFF_p f|$, in general, depends on $p$ (however, this does not happen on $\PI$ spaces, see e.g.\ \cite[Appendix C]{ACM14} or \cite[Theorem A.9]{Bjorn-Bjorn11}). 
\\

\noindent\textbf{PI spaces.}\\We recall the definition of the class of $\PI$ spaces.
\begin{defn}
Let \((\XX,\dist,\mass)\) be a metric measure space. Then:
\begin{itemize}
\item[\(\rm i)\)] We say that \((\XX,\dist,\mass)\) is {uniformly} doubling if 
there exists a constant \(C_D>0\) such that
\begin{equation}\label{fulldoubing}
\mass(B_{2r}(x))\leq C_D\mass(B_r(x))\qquad\text{for every }x\in\XX\text{ and }r>0.    
\end{equation}
\item[\(\rm ii)\)] For $p\in (1,\infty)$, we say that \((\XX,\dist,\mass)\) supports a \((1,p)\)-Poincar\'{e} inequality if there exists a constant \(C_P>0\) such that
\[
\dashint_{B_r(x)}\Big|f-\dashint_{B_r(x)}f\dd\mass\Big|\dd\mass\leq C_P r\bigg(\dashint_{B_{ r}(x)}\lip(f)^p\dd\mass\bigg)^{1/p}
\qquad\text{for every }x\in\XX\text{ and }r>0,
\]
whenever $f\in\LIP_{\mathrm{bs}}(\XX)$.
\item[\(\rm iii)\)] We say that \((\XX,\dist,\mass)\) is a $p$-$\PI$ space if it is uniformly doubling and
it supports a \((1,p)\)-Poincar\'{e} inequality, i.e.\ if it satisfies items $\rm i)$ and $\rm ii)$ above. If the value of $p\in (1,\infty)$ is not important, we simply say that \((\XX,\dist,\mass)\) is a $\PI$ space. 
\end{itemize}
\end{defn}
	\begin{rem}\label{verodoub}
 We will often need Lemma \ref{lemma28} below, which is taken from \cite{DiMarSquas}. We remind that, to apply the lemma to a metric measure space $(\XX,\dist,\mass)$, we need the full doubling condition on the reference measure, i.e.\ the existence of some $C_D>0$ satisfying \eqref{fulldoubing}. In particular, the local uniform doubling property, that asks that for every $R>0$ there exists $C_D(R)>0$ such that 
		\begin{equation}
			\mass(B_{2 r}(x))\le C_D(R)\mass(B_r(x))\qquad\text{for every $x\in\XX$ and $r\in(0,R)$}
		\end{equation}
		is not sufficient for our purposes. 
  This is the reason why we asked the validity of \eqref{fulldoubing} in the definition of $\PI$ space, even though it is now customary to include in the class of $\PI$ spaces even spaces that are only locally uniformly doubling. Similar considerations hold for what concerns the statement of the Poincar\'{e} inequality, see the assumptions in \cite{DiMarSquas}.

Nevertheless, to prove our main Theorem \ref{thmlimits}, a weaker form of the $(1,p)$-Poincar\'{e} inequality is enough: it is enough to ask that 
there exist a constant \(\lambda\geq 1\)  and, for any radius \(R>0\), a constant \(C_P=C_P(R)>0\), such that
\[
\dashint_{B_r(x)}\Big|f-\dashint_{B_r(x)}f\dd\mass\Big|\dd\mass\leq C_P r\bigg(\dashint_{B_{\lambda r}(x)}\lip(f)^p\dd\mass\bigg)^{1/p}
\qquad\text{for every }x\in\XX\text{ and }r\in(0,R),
\]
whenever $f\in\LIP_{\mathrm{bs}}(\XX)$. This is because the Poincar\'{e} inequality is only needed to employ the results of \cite{Cheeger00} (in \cite{Cheeger00} the standing assumption, besides the doubling condition on the measures, was this weaker form of the Poincar\'{e} inequality):
\begin{itemize}
    \item in the blow-up argument of Lemma \ref{blowup} and Lemma \ref{blowuph}, we need the characterization of blow-ups of Lipschitz functions as generalized linear functions,
    \item in the proof of Theorem \ref{thmlimits}, we need the identification between local Lipschitz constant and minimal weak upper gradient, see \eqref{eq:|Df|=lipf} below.
 \end{itemize} 
 Notice that, however,  the uniform doubling condition is still needed to prove Theorem \ref{thmlimits} (see the proof of Proposition \ref{mainprop}). 
\fr
	\end{rem}
 \begin{lem}[{\cite[Lemma 2.8]{DiMarSquas}}]\label{lemma28}
 Let $(\XX,\dist,\mass)$ be a uniformly doubling metric measure space. Then, for every $x\in \XX$ and $R>0$,
\begin{equation}\label{lemma28eq}
\int_{\XX\setminus B_{R}(x)} \frac{1}{\mass(B_{\dist(x,y)}(x))\dist(x,y)^p}\dd\mass(y)\le  \frac{C}{R^p},
\end{equation}
where the constant $C>0$ depends only on $(\XX,\dist,\mass)$.
 \end{lem}
 
A key property of $p$-$\PI$ spaces, proved in \cite{Cheeger00}, is that  
\begin{equation}\label{eq:|Df|=lipf}
|\DIFF_p f|=\lip(f)\quad \mass\text{-a.e.}\qquad\text{ for every }f\in\LIP_{\rm bs}(\XX).
\end{equation}
\\

\noindent\textbf{pmGH topology.} \\
A fundamental concept in the blow-up argument of this note is the pointed-measured-Gromov--Hausdorff (pmGH) convergence (see e.g.\ \cite{GMS15} and the references therein).
Take a sequence \(\big((\XX_n,\dist_n,\mass_n,x_n)\big)_n\) of pointed uniformly doubling spaces, sharing the same constant $C_D$ of \eqref{fulldoubing}.
Then we say that \((\XX_n,\dist_n,\mass_n,x_n)\) converges to the pointed metric measure space \((\XX_\infty,\dist_\infty,\mass_\infty,x_\infty)\)
in the pmGH sense provided that there exist a proper metric space \((\ZZ,\dist_\ZZ)\) and isometric embeddings \(\iota_n\colon(\XX_n,\dist_n)\hookrightarrow(\ZZ,\dist_\ZZ)\) for
\(n\in\NN\cup\{\infty\}\) such that
\[
(\iota_n)_\#\mass_n\rightharpoonup(\iota_\infty)_\#\mass_\infty\qquad\text{in duality with }C_{\rm bs}(\ZZ),
\]
and \(\iota_n(x_n)\to\iota_\infty(x_\infty)\). If we have a sequence of functions $f_n:\XX_n\rightarrow\RR$, for $n\in\NN\cup\{\infty\}$, we say that $f_n$ converges to $f$ locally uniformly (with respect to the realization $(\ZZ,\dist_\ZZ)$) if for every $R,\epsilon>0$, there exist $\bar n\in\NN$ and $\delta>0$ such that 
$$
|f_n(x)-f_\infty(x')|<\epsilon\qquad\text{if $x'\in B_R(x_\infty)\subseteq\XX_\infty$, $x\in \XX_n$, where $n\ge \bar n$ and $\dist_\ZZ(\iota_n(x),\iota_\infty(x'))<\delta$.}
$$
By Gromov compactness Theorem, any family of pointed uniformly doubling spaces \(\big((\XX_n,\dist_n,\mass_n,x_n)\big)_n\), sharing the same constant $C_D$ of \eqref{fulldoubing} and such that $$\sup_n\mass_n(B_1(x))<\infty$$ is precompact with respect to the pmGH topology. 

Given a pointed metric measure space \((\XX,\dist,\mass,x)\)
and a radius \(r>0\), we define the normalized measure \(\mass_r^x\) as
\[
\mass_r^x\coloneqq\frac{\mass}{\mass(B_r(x))}.
\]
The tangent cone \({\rm Tan}_x(\XX,\dist,\mass)\)
is then defined as the family of all those pointed metric measure spaces \((\YY,\dist_\YY,\mass_\YY,y)\) such that
\((\XX,r_i^{-1}\dist,\mass_{r_i}^x,x)\to(\YY,\dist_\YY,\mass_\YY,y)\) in the pmGH sense for some sequence \((r_i)_i\) with \(r_i\searrow 0\).
Thanks to the compactness property of the pmGH convergence, we have \({\rm Tan}_x(\XX,\dist,\mass)\neq\emptyset\) for every uniformly doubling metric measure space $(\XX,\dist,\mass)$ and for every $x\in\XX$. We define $\underline\LL^n$ to be the normalized Lebesgue measure, i.e.\
\begin{equation}\label{normalized1}
    \underline\LL^n=(\LL^n)_1^0=\frac{1}{\omega_n}\LL^n,
\end{equation}
where $\omega_n\defeq\LL^n(B_1(0))$.
\\

\noindent\textbf{Heisenberg group} \\
We represent the Heisenberg group $\heis^{n}$ as the set $\mathbb{C}^n\times\RR \equiv \RR^{2n+1}$; thus its points will be described as $x=(z,t)=(\zeta+i\eta,t)=(\zeta,\eta,t)$, with $z\in\mathbb{C}^n$, $\zeta,\eta\in\RR^n$, $t\in\RR$. The group operation on $\heis^{n}$ will be defined as follows: whenever $x=(z,t)\in \heis^{n}$ and $x'=(z',t')\in \heis^{n}$,
	\begin{equation}\label{equation: group operation}
	x\,\cdot\, x'  \coloneqq \left(z+z', t+t'+ 2 \mathrm{Im}\left(\langle z,\bar{z'}\rangle\right)\right),
	\end{equation}
where $\bar{z}\in\mathbb{C}$ denotes the conjugate of $z\in \mathbb{C}$.
As a consequence, it is easy to verify that the group identity is the origin 0 and the inverse of a point is given by $(z,t)^{-1}=(-z,-t)$.
We also introduce the following family of \emph{non-isotropic dilations}, since they are fundamental in the geometry of the Heisenberg group: if $x=(z,t)\in \heis^{n}$  and $\lambda>0$,
	\begin{equation}\label{equation: dilations}
	\delta_{\lambda}(x)\coloneqq(\lambda z,\lambda^2t). 
	\end{equation}
The Heisenberg group $\heis^n$ admits the structure of a Lie group of topological dimension $2n+1$. Its Lie algebra $\mathfrak{h}_n$ of left invariant vector fields is  stratified of step $2$ and it is (linearly) generated by $(2n+1)-$vector fields
$X_1,\dots,X_{n}, Y_1,\dots, Y_n, T$. 
We recall that an absolutely continuous curve $\gamma\colon [a,b]\to \heis^n$ is \emph{horizontal} if there exist $u_{1}, \ldots, u_{2n}\in L^{1}[a,b]$ such that for almost every $t\in [a,b]$:
\[\gamma'(t)=\sum_{i=1}^{n}\left(u_{i}(t)X_{i}(\gamma(t))+u_{i+n}(t)Y_{i}(\gamma(t))\right).\]
Define the horizontal length of such a curve $\gamma$ by $L(\gamma)=\int_{a}^{b}|u(t)|\dd t$, where $u=(u_{1}, \ldots, u_{2n})$ and $|\,\cdot\,|$ denotes the Euclidean norm on $\RR^{2n}$.
The \emph{Carnot-Carath\'eodory distance (CC distance)} between points $x, y \in \heis^n$ is defined by:
\[d_{\mathbb{H}}(x,y)=\inf \{ L(\gamma) : \gamma \colon [0,1]\to \heis^n \mbox{ horizontal joining }x\mbox{ to }y \}.\]
The Chow-Rashevskii Theorem implies that any two points of $\heis^n$ can be connected by horizontal curves \cite[Theorem 9.1.3]{BLU07}. It follows that the CC distance is indeed a distance on $\heis^n$. 
For brevity we write $\|x\|$ instead of $d_{\mathbb{H}}(x,0)$. It is well known that $d_{\mathbb{H}}$ induces the same topology as the Euclidean distance but the two are not Lipschitz equivalent.\\
The Haar measure of the $\mathbb{H}^n$ is given by the Lebesgue measure $\LL^{2n+1}$. Unless otherwise stated, when
saying that a property holds a.e.\ on $\heis^{n}$; we will always mean that it holds a.e.\ with
respect to $\LL^{2n+1}$.
As in the Euclidean case, we define $\underline{\LL}^{2n+1}$ to be the normalized Lebesgue measure, i.e.
\begin{equation}\label{normalized2}
    \underline{\LL}^{2n+1}=\frac{1}{\omega_n}\LL^{2n+1}
\end{equation}
where $\omega_n=\LL^{2n+1}(B_1(0))$ and $B_1(0)$ denotes the unit ball with respect to $d_{\mathbb{H}}$.

\section{Main part}
\subsection{Auxiliary results}
\noindent\textbf{The functionals $I_\delta$ and $J_\delta$}
	\begin{defn}\label{IandJ}Let $p\in(1,\infty)$ and let $f\in L^p(\mass)$. We define, for $\delta>0$,
		\begin{equation}
			I_\delta(f)\defeq\iint_{|f(x)-f(y)|>\delta}\frac{\delta^p}{\mass(B_{\dist(x,y)}(x))\dist(x,y)^p}\dd\mass(y)\dd\mass(x)
		\end{equation}
		and 
				\begin{equation}
			J_\delta(f)\defeq\iint_{|f(x)-f(y)|\le 1}\frac{\delta|f(x)-f(y)|^{p+\delta}}{\mass(B_{\dist(x,y)}(x))\dist(x,y)^p}\dd\mass(y)\dd\mass(x).
		\end{equation}
	\end{defn}
The following lemma is morally \cite[Equation (3.6)]{NguyenPinamontiSquassinaVecchi} and follows from the definition of $I_\delta$ together with the fact that the triangle inequality implies 
 	\begin{equation}\notag
		\chi_{\{|f(x)-f(y)|> \delta\}}\le	\chi_{\{|g(x)-g(y)|> (1-\epsilon)\delta\}}+ \chi_{\{|(f-g)(x)-(f-g)(y)|>\epsilon\delta\}}
	\end{equation}
 for every $\delta,\epsilon>0$.
 \begin{lem}\label{tricklem}
	Let $p\in(1,\infty)$ and let $f,g\in L^p(\mass)$. Then, if $\delta,\epsilon>0$,
\begin{equation}\label{trick}
		I_\delta(f)\le \frac{1}{(1-\epsilon)^p}  I_{(1-\epsilon)\delta}(g)+\frac{1}{\epsilon^p} I_{\epsilon\delta}(f-g).
\end{equation}
	
\end{lem}

The following lemma morally follows from an easy computation done during the proof of \cite[Theorem 1.5]{DiMarSquas}, and we are going to use it also in the spirit of  \cite[Lemma 2.3]{CLL14}.
\begin{lem}\label{IandJlem}
Let $p\in(1,\infty)$ and let $f\in L^p(\mass)$. Then, for every $\epsilon\in(0,1)$,
\begin{equation}\label{IandJeq}
	\int_0^1\epsilon\delta^{\epsilon-1}I_\delta(f)\dd\delta=	\frac{1}{p+\epsilon}\bigg(J_\epsilon(f) +	\epsilon\iint_{|f(x)-f(y)|>  1}\frac{1}{\mass(B_{\dist(x,y)}(x))\dist(x,y)^p}\dd\mass(y)\dd\mass(x)\bigg).
\end{equation}
\end{lem}

\begin{proof}
Notice first that,	
\begin{equation}
\begin{split}
	&\int_0^1 \epsilon\delta^{p+\epsilon-1}\chi_{\{|f(x)-f(y)|>\delta\}}(x,y)\dd\delta=\frac{\epsilon}{p+\epsilon}(|f(x)-f(y)|\wedge 1)^{p+\epsilon}\\
	&\qquad=\frac{1}{p+\epsilon}\Big(\epsilon\chi_{\{|f(x)-f(y)|\le 1\}}(x,y) |f(x)-f(y)|^{p+\epsilon} +\epsilon \chi_{\{|f(x)-f(y)|>  1\}}(x,y)\Big),
\end{split}
\end{equation}
then \eqref{IandJeq}  follows from Fubini's Theorem.
\end{proof}
\medskip\noindent\textbf{The blow-up argument}
\begin{lem}\label{blowup}
	Let $(\XX,\dist,\mass)$ be a $\PI$ space satisfying Assumption \ref{Ass1}.
Let $f\in\LIP(\XX)$. Then, for every $p\in(1,\infty)$,
\begin{equation}\label{blowupeq}
	\lim_{\delta\searrow 0}\int_\XX \chi_{\{|f(y)-f(x)|>\delta\}}(x,y)\frac{\delta^p}{\mass(B_{\dist(x,y)}(x))\dist(x,y)^p}\dd\mass(y)=\frac{C_{n,p}}{p \omega_n} \lip(f)^p(x)\qquad\text{for $\mass$-a.e.\ $x$},
\end{equation}
	where $C_{n,p}$ is as in  Definition \ref{defnormconst}.
\end{lem}
\begin{proof}
We prove that \eqref{blowupeq} holds at any fixed point $x$ satisfying the conclusions of Assumption \ref{Ass1} and the conclusion of \cite[Theorem 10.2]{Cheeger00}. Up to rescaling, we can assume that $f$ is $1$-Lipschitz.
Fix any sequence $\delta_k\searrow 0$, it is enough to show that we can extract a (non relabelled) subsequence such that the limit in \eqref{blowupeq} holds for this subsequence. 

Thanks to Gromov compactness Theorem, Arzelà--Ascoli Theorem and the choice of $x$ (in particular, \eqref{tangA}), we can assume (up to taking a non relabelled subsequence) that 
$$
(\XX,\dist_{\delta_k},\mass_{\delta_k}^x,x)\defeq(\XX,\delta_k^{-1}\dist,\mass/\mass(B_{\delta_k}(x)),x)\rightarrow (\RR^n,\dist_e,\underline{\LL}^n,0)\qquad\text{in the $\rm pmGH$ topology}
$$ 
and that $f_k\defeq \delta_k^{-1}(f-f(x))$ converge locally uniformly to a linear function $L_x:\RR^n\rightarrow\RR$,  where, for some $\zeta\in S^{n-1}$, $$L_x(y)=\lip(f)(x)(y\,\cdot\,\zeta)\qquad\text{for every }y\in\RR^n$$
 (see \cite[Proposition 3.3]{Gorny} or \cite[Lemma 3.3]{BPP}), with respect to a realization $(\ZZ,\dist_\ZZ)$ and isometric embeddings $(\iota_k)_k$.

	We let $C>0$ denote a constant that depends only $(\XX,\dist,\mass)$ and may vary during the proof. If $R>0$, by \eqref{lemma28eq}, for every $k$,
\begin{equation}\label{aa}
\begin{split}
&\int_{\XX\setminus B_{\delta_k R}(x)} \chi_{\{|f(y)-f(x)|>\delta_k\}}(x,y)\frac{\delta_k^p}{\mass(B_{\dist(x,y)}(x))\dist(x,y)^p}\dd\mass(y)\\&\qquad\le \delta_k^p\int_{\XX\setminus B_{\delta_k R}(x)} \frac{1}{\mass(B_{\dist(x,y)}(x))\dist(x,y)^p}\dd\mass(y)\\&\qquad\le  \delta_k^p  \frac{C}{(\delta_k R)^p}=\frac{C}{R^p}.
\end{split}
\end{equation}

Now, for $\epsilon\in (0,1)$ we define $\phi_\epsilon\in\LIP(\RR)$ by
$$
\phi_\epsilon(x)\defeq
\begin{cases}
	0\qquad&\text{if }x\le 1-\epsilon,\\
	x/\epsilon+(\epsilon-1)/\epsilon\qquad&\text{if }1-\epsilon\le x\le 1,\\
	1\qquad&\text{if }1\le x;
\end{cases}
$$
and we define
$\psi_\epsilon\in\LIP(\RR)$ by
$$
\psi_\epsilon(x)\defeq
\begin{cases}
	0\qquad&\text{if }x\le 1-2\epsilon,\\
	x/\epsilon+(2\epsilon-1)/\epsilon\qquad&\text{if }1-2\epsilon\le x\le 1-\epsilon,\\
	1\qquad&\text{if }1-\epsilon\le x\le 1,\\
	-x/\epsilon+(\epsilon+1)/\epsilon\qquad&\text{if }1 \le x\le 1+\epsilon,\\
	0\qquad&\text{if }1+\epsilon\le x.
\end{cases}
$$
Notice that for every $\epsilon$, $\phi_\epsilon\circ|f_k|$ converge locally uniformly to $\phi_\epsilon\circ |L_x|$ and similarly for $\psi_\epsilon$ in place of $\phi_\epsilon$, as $f_k$ converge locally uniformly to $L_x$.

Now we compute, being $f$ $1$-Lipschitz,
\begin{equation}\label{omkn}
\begin{split}
		&\int_{B_{\delta_k R}(x)} \chi_{\{|f(y)-f(x)|>\delta_k\}}(x,y)\frac{\delta_k^p}{\mass(B_{\dist(x,y)}(x))\dist(x,y)^p}\dd\mass(y)
	\\&\qquad =\int_{B_{\delta_k R}(x)\setminus B_{\delta_k }(x)} \chi_{\{|f(y)-f(x)|>\delta_k\}}(x,y)\frac{\delta_k^n}{\dist(x,y)^n\mass(B_{\delta_k}(x))}\frac{\delta_k^p}{\dist(x,y)^p}\dd\mass(y)+\omega^{\rm I}_{k,R},
\end{split}
\end{equation}
where 
$$
\omega^{\rm I}_{k,R}\defeq \int_{B_{\delta_k R}(x)\setminus B_{\delta_k }(x)} \chi_{\{|f(y)-f(x)|>\delta_k\}}(x,y)\bigg(\frac{1}{\mass(B_{\dist(x,y)}(x))}-\frac{\delta_k^n}{\dist(x,y)^n\mass(B_{\delta_k}(x))}\bigg)\frac{\delta_k^p}{\dist(x,y)^p}\dd\mass(y).
$$
Then, recalling \eqref{lemma28eq}  and then  \eqref{densityA},
\begin{equation}\label{est1}
\begin{split}
		|\omega^{\rm I}_{k,R}|&\le \int_{B_{\delta_k R}(x)\setminus B_{\delta_k }(x)} \bigg|1-\frac{\delta_k^n\mass(B_{\dist(x,y)}(x))}{\dist(x,y)^n\mass(B_{\delta_k}(x))}\bigg|\frac{\delta_k^p}{\mass(B_{\dist(x,y)}(x))\dist(x,y)^p}\dd\mass(y)
	\\&\le \sup_{y\in B_{\delta_kR}(x)} \bigg|1-\frac{\delta_k^n\mass(B_{\dist(x,y)}(x))}{\dist(x,y)^n\mass(B_{\delta_k}(x))}\bigg|\int_{\XX\setminus B_{\delta_k }(x)}\frac{\delta_k^p}{\mass(B_{\dist(x,y)}(x))\dist(x,y)^p}\dd\mass(y)\rightarrow 0.
\end{split}
\end{equation}

Therefore we reduce ourselves to compute the limit as $k\rightarrow\infty$ of
\begin{equation}\label{omkn1}
\begin{split}
	&\int_{B_{\delta_k R}(x)\setminus B_{\delta_k}(x)}\chi_{\{|f(y)-f(x)|>\delta_k\}}(x,y)\frac{\delta_k^{n+p}}{\dist(x,y)^{n+p}}\frac{1}{\mass(B_{\delta_k}(x))}\dd\mass(y)\\&\qquad=\int_{B^k_R(x)\setminus B_1^k(x)} \chi_{\{|f_k|>1\}}(y)\frac{1}{\dist_{\delta_k}(x,y)^{n+p}}\dd\mass_{\delta_k}^x(y)
\\&\qquad=\int_{B^k_R(x)\setminus B_1^k(x)} \phi_\epsilon(|f_k(y)|)\frac{1}{\dist_{\delta_k}(x,y)^{n+p}}\dd\mass_{\delta_k}^x(y)+ \omega^{\rm II}_{k,R,\epsilon},
\end{split}
\end{equation}
where 
$$
\omega^{\rm II}_{k,R,\epsilon}\defeq\int_{B^k_R(x)\setminus B_1^k(x)} \big(\chi_{\{|f_k|>1\}}(y)-\phi_\epsilon(|f_k(y)|)\big)\frac{1}{\dist_{\delta_k}(x,y)^{n+p}}\dd\mass_{\delta_k}^x(y).
$$
We can estimate
\begin{equation}
\begin{split}\label{est2}
	|\omega^{\rm II}_{k,R,\epsilon}|&\le\int_{B^k_R(x)\setminus B_1^k(x)} \big|\chi_{\{|f_k|>1\}}(y)-\phi_\epsilon(|f_k(y)|)\big|\frac{1}{\dist_{\delta_k}(x,y)^{n+p}}\dd\mass_{\delta_k}^x(y)
\\&\le\int_{B^k_R(x)\setminus B_1^k(x)} \psi_\epsilon(|f_k(y)|)\frac{1}{\dist_{\delta_k}(x,y)^{n+p}}\dd\mass_{\delta_k}^x(y).
\end{split}
\end{equation}

Hence we know from \eqref{omkn} and \eqref{omkn1} that
\begin{equation}\label{csdsca}
\begin{split}
		&\limsup_k \bigg|\int_{B_{\delta_k R}(x)} \chi_{\{|f(y)-f(x)|>\delta_k\}}(x,y)\frac{\delta_k^p}{\mass(B_{\dist(x,y)}(x))\dist(x,y)^p}\dd\mass(y)\\&\qquad\qquad\qquad\qquad\qquad-\int_{B^k_R(x)\setminus B_1^k(x)} \phi_\epsilon(|f_k(y)|)\frac{1}{\dist_{\delta_k}(x,y)^{n+p}}\dd\mass_{\delta_k}^x(y)\bigg|\\&\qquad\le\limsup_k\big(|\omega^{\rm I}_{k,R}|+|\omega^{\rm II}_{k,R,\epsilon}|\big).
\end{split}
\end{equation}

We compute, by local uniform convergence,
\begin{align}\label{conv1}
	\lim_k\int_{B^k_R(x)\setminus B_1^k(x)} \phi_\epsilon(|f_k(y)|)\frac{1}{\dist_{\delta_k}(x,y)^{n+p}}\dd\mass_{\delta_k}^x(y)=\int_{B_R(0)\setminus B_1(0)} \phi_\epsilon(|L_x(y)|)\frac{1}{|y|^{n+p}}\dd\underline{\LL}^n(y)
\end{align}
and, similarly,
\begin{align}\label{conv2}
	\lim_k\int_{B^k_R(x)\setminus B_1^k(x)} \psi_\epsilon(|f_k(y)|)\frac{1}{\dist_{\delta_k}(x,y)^{n+p}}\dd\mass_{\delta_k}^x(y)=\int_{B_R(0)\setminus B_1(0)} \psi_\epsilon(|L_x(y)|)\frac{1}{|y|^{n+p}}\dd\underline{\LL}^n(y).
\end{align}
We briefly justify this claim. We can put ourselves in a situation in which we have sequence of functions $g_k$ that locally uniformly to $g_\infty$, where $g_\infty$ is Lipschitz. We then extend $g_\infty$ to a Lipschitz function on $(\ZZ,\dist_\ZZ)$ that we call $g'_\infty$, in the sense that $g_\infty'\circ\iota_\infty=g_\infty$. Then, if $\varphi\in C_{\mathrm{bs}}(\ZZ)$, 
\begin{align}
    \int_{\XX} \varphi g_k\dd\mass^x_{\delta_k}=    \int_{\XX} \varphi g_\infty'\dd\mass^x_{\delta_k}+\int_{\XX} \varphi (g_k-g_\infty')\dd\mass^x_{\delta_k}.
\end{align}
Therefore, the claim follows by the weak convergence of $\mass^x_{\delta_k}$ to $\underline{\LL}^n$ provided we show that, 
$$
\int_{B_{R}^k(x)} |g_k-g_\infty'|\dd\mass_{\delta_k}^x\rightarrow 0.
$$
By local uniform convergence of $g_k$ to $g_\infty$ and the Lipschitz property of $g_\infty'$, we have that $|g_k-g_\infty'|$ locally uniformly converges to $0$, hence it is enough to show that for every $\tau>0$, there exists $\bar k\in\NN$ such that for every $k\ge \bar k$ and $y\in B_{\delta_k (R+1)}(x)$, there exists $y'\in B_{R+2}(0)$ with $\dist_\ZZ(\iota_k(y),\iota_\infty(y'))<\tau$. This is easily proved by contradiction, relying on standard arguments.

Therefore \eqref{csdsca} reads, taking into account also \eqref{est1} and \eqref{est2},
\begin{equation}
	\begin{split}
		&\limsup_k \bigg|\int_{B_{\delta_k R}(x)} \chi_{\{|f(y)-f(x)|>\delta_k\}}(x,y)\frac{\delta_k^p}{\mass(B_{\dist(x,y)}(x))\dist(x,y)^p}\dd\mass(y)\\&\qquad\qquad\qquad\qquad\qquad\qquad-\int_{B_R(0)\setminus B_1(0)} \phi_\epsilon(|L_x(y)|)\frac{1}{|y|^{n+p}}\dd\underline{\LL}^n(y)
		\bigg|
		\\&\qquad\le\int_{B_R(0)\setminus B_1(0)} \psi_\epsilon(|L_x(y)|)\frac{1}{|y|^{n+p}}\dd\underline{\LL}^n(y).
	\end{split}
\end{equation}
We can let now $\epsilon\searrow 0$ to infer that
\begin{equation}
	\begin{split}
		&\limsup_k \bigg|\int_{B_{\delta_k R}(x)} \chi_{\{|f(y)-f(x)|>\delta_k\}}(x,y)\frac{\delta_k^p}{\mass(B_{\dist(x,y)}(x))\dist(x,y)^p}\dd\mass(y)\\&\qquad\qquad\qquad\qquad\qquad\qquad-\int_{B_R(0)\setminus B_1(0)} \chi_{\{|L_x|\ge 1\}}(y)\frac{1}{|y|^{n+p}}\dd\underline{\LL}^n(y)
		\bigg|
		\\&\qquad\le\int_{B_R(0)\setminus B_1(0)} \chi_{\{|L_x|= 1\}}(y)\frac{1}{|y|^{n+p}}\dd\underline{\LL}^n(y)=0,
	\end{split}
\end{equation}
being the set $\{|L_x|=1\}$ either empty of the union of two hyperplanes.

All in all, recalling also \eqref{aa}, we infer that
\begin{equation}\label{mkcds}
\begin{split}
		&\lim_k \int_\XX \chi_{\{|f(y)-f(x)|>\delta_k\}}(x,y)\frac{\delta_k^p}{\mass(B_{\dist(x,y)}(x))\dist(x,y)^p}\dd\mass(y)\\&\qquad=\int \chi_{\{|L_x|\ge 1\}}(y)\frac{1}{|y|^{n+p}}\dd\underline{\LL}^n(y)=\frac{1}{\omega_n}\int_{\{y:\lip(f)(x)|y\,\cdot \,\zeta|\geq 1\}}\frac{1}{|y|^{n+p}}\dd{\LL}^n(y).
\end{split}
\end{equation}
Now, if $\lip(f)(x)\ne 0$, we perform the change of variables $z=y\lip(f)(x)$ and we see that 
\begin{equation}
\begin{split}
	&\int_{\{y:\lip(f)(x)|y\,\cdot\,\zeta|\ge 1\}}\frac{1}{|y|^{n+p}}\dd{\LL}^n(y)=\lip(f)^p(x)\int_{\{z:|z\,\cdot\,\zeta|\ge 1\}}\frac{1}{|z|^{n+p}}\dd\LL^n(z)\\&\qquad=\lip(f)^p(x)\int_{S^{n-1}}\int_0^\infty \frac{\rho^{n-1}}{\rho^{n+p}}\chi_{\{\rho|\hat z\,\cdot\,\zeta|\ge 1\}}\dd \rho\dd\HH^{n-1}(\hat z)
\\&\qquad=  \lip(f)^p(x)\int_{S^{n-1}}\int_{1/|\hat z\,\cdot\,\zeta|}^\infty \rho^{-p-1}\dd \rho\dd\HH^{n-1}(\hat z)
\\&\qquad=  \lip(f)^p(x)\int_{S^{n-1}}\frac{|\hat z\,\cdot\,\zeta|^p}{p}\dd\HH^{n-1}(\hat z),
\end{split}
\end{equation}
where we used polar coordinates.  Hence, recalling the definition of $C_{n,p}$ in  Definition \ref{defnormconst},
\begin{equation}\label{csdacsa}
	\int_{\{y:\lip(f)(x)|y\,\cdot\,\zeta|\ge 1\}}\frac{1}{|y|^{n+p}}\dd{\LL}^n(y)= \lip(f)^p(x)\frac{C_{n,p}}{p}.
\end{equation}
Notice that \eqref{csdacsa} continues to holds even if $\lip(f)(x)=0$.
Then, \eqref{blowupeq} follows from \eqref{mkcds} and \eqref{csdacsa}.
\end{proof}
\begin{lem}\label{blowuph}
	Let $(\XX,\dist,\mass)$ be a $\PI$ space satisfying Assumption \ref{Ass2}.
Let $f\in\LIP(\XX)$. Then, for every $p\in(1,\infty)$,
\begin{equation}\label{blowupeq2}
	\lim_{\delta\searrow 0}\int_\XX \chi_{\{|f(y)-f(x)|>\delta\}}(x,y)\frac{\delta^p}{\mass(B_{\dist(x,y)}(x))\dist(x,y)^p}\dd\mass(y)=\frac{ C_{n,p}}{p \omega_n} \lip(f)^p(x)\qquad\text{for $\mass$-a.e.\ $x$},
\end{equation}
	where $C_{n,p}$ is as in Definition \ref{defnormconst}.
\end{lem}
\begin{proof}
The proof follows verbatim the one of Lemma \ref{blowup} up to some technical details that we clarify below.
As in Lemma \ref{blowup}, up to rescaling, we can assume that $f$ is $1$-Lipschitz.
Fix any sequence $\delta_k\searrow 0$, it is enough to show that we can extract a (non relabelled) subsequence such that the limit in \eqref{blowupeq} holds for this subsequence. 

Thanks to Gromov compactness Theorem, Arzelà--Ascoli Theorem and the choice of $x$ (in particular, \eqref{cdsasacds}), we can assume (up to taking a non relabelled subsequence) that 
$$
(\XX,\dist_{\delta_k},\mass_{\delta_k}^x,x)\defeq(\XX,\delta_k^{-1}\dist,\mass/\mass(B_{\delta_k}(x)),x)\rightarrow (\mathbb{H}^n,\dist_{\mathbb{H}},\underline{\LL}^{2n+1},0)\qquad\text{in the $\rm pmGH$ topology}
$$ 
and that $f_k\defeq \delta_k^{-1}(f-f(x))$ converge locally uniformly to a linear function $L_x:\mathbb{H}^n\rightarrow\RR$,  where, for some $\zeta=(\zeta_1,\ldots,\zeta_{2n},0)$ with $(\zeta_1,\ldots,\zeta_{2n})\in \mathbb{S}^{2n-1}$, $$L_x(y)=\lip(f)(x)(y\,\cdot\,\zeta)\qquad\text{for every }y\in\mathbb{H}^n$$
 (see \cite[Theorem 4.2]{Gorny}). 
Following \eqref{omkn} and \eqref{est1} of the proof of Lemma \ref{blowup}  (here the assumption \eqref{density2} comes into play),
we reduce ourselves to compute the limit of  $$\int_{B_{\delta_k R}(x)\setminus B_{\delta_k }(x)} \chi_{\{|f(y)-f(x)|>\delta_k\}}(x,y)\frac{\delta_k^{2n+2}}{\dist(x,y)^{2n+2}\mass(B_{\delta_k}(x))}\frac{\delta_k^p}{\dist(x,y)^p}\dd\mass(y).$$
Hence, employing the same computations as in the proof of Lemma \ref{blowup}, we are left to compute
\[\frac{1}{\omega_n}\int_{\{y\in\mathbb{H}^n:\lip(f)(x)|y\,\cdot\,\zeta|\geq 1\}}\frac{1}{\|y\|^{Q+p}}\dd{\LL}^{2n+1}(y)\]
where $Q\defeq2n+2$. Now, if $\lip(f)(x)\ne 0$, we perform the change of variables $z=\delta_{\lip(f)(x)}y$ and we see that
\begin{align}
    \int_{\{y:\lip(f)(x)|y\,\cdot\, \zeta|\geq 1\}}\frac{1}{\|y\|^{Q+p}}\dd{\LL}^{2n+1}(y)&=\lip(f)(x)^{p}\int_{\{z:|z\,\cdot\, \zeta|\geq 1\}}\frac{1}{\|z\|^{Q+p}}\dd{\LL}^{2n+1}(z)\\
    &=\lip(f)(x)^{p}\int_{\mathbb{H}^n}\chi_{\{z:|z\,\cdot\, \zeta|\geq 1\}}\frac{1}{\|z\|^{Q+p}}\dd{\LL}^{2n+1}(z).
\end{align}
Set $z=\delta_h\sigma$ where $\sigma\in \Sigma\defeq\{u\in \mathbb{H}^n:\|u\|=1\}$ and $h\in [0,\infty)$. Then
\begin{align}
    \lip(f)(x)^{p}\int_{\mathbb{H}^n}\chi_{\{z:|z\,\cdot\, \zeta|\geq 1\}}\frac{1}{\|z\|^{Q+p}}\dd{\LL}^{2n+1}(z)&=\lip(f)(x)^{p}\int_{\Sigma}\int_{0}^{\infty}\chi_{\{h|v\,\cdot\, \zeta|\geq 1\}}\frac{h^{Q-1}}{h^{Q+p}}\dd h\dd \sigma(v)\\
    &=\lip(f)(x)^{p}\int_{\Sigma}\int_{0}^{\infty}\chi_{\{h|v\,\cdot\, \zeta|\geq 1\}}\frac{1}{h^{p+1}}\dd h\dd \sigma(v)\\
    &=\lip(f)(x)^{p}\int_{\Sigma}\int_{1/|v\,\cdot\, \zeta|}^{\infty}\frac{1}{h^{p+1}}\dd h\dd \sigma(v)\\
    &=\frac{\lip(f)(x)^{p}}{p}\int_{\Sigma}|v\,\cdot\, \zeta|^p\dd \sigma(v),
\end{align}
where $\dd \sigma$ is a suitable Radon measure on $\Sigma$, cfr. \cite[Proposition 1.15]{FolSte}. The proof is concluded.
\end{proof}

\medskip\noindent\textbf{Application of the blow-up argument}
\begin{prop}\label{mainprop}
	Let $(\XX,\dist,\mass)$ be a $\PI$ space satisfying either Assumption \ref{Ass1} or Assumption \ref{Ass2}.	
	Let $f\in\LIP_{\mathrm{bs}}(\XX)$. Then, for every $p\in(1,\infty)$,
\begin{equation}\label{mainpropeq}
		\lim_{\delta\searrow 0} I_\delta(f)=\frac{C_{n,p}} {p \omega_n}\int_\XX \lip(f)^p\dd\mass,
\end{equation}
	where $C_{n,p}$ is as in Definition \ref{defnormconst}.
\end{prop}
\begin{proof}
	We let $C>0$ denote a constant that depends only $(\XX,\dist,\mass)$ and may vary during the proof.
%
%
	By scaling, we can assume that $f$ is $1$-Lipschitz, so that, by \eqref{lemma28eq}, for every $x\in\XX$,
	\begin{equation}\label{csdcs}
\begin{split}
				&\int_\XX\chi_{\{|f(x)-f(y)|>\delta\}}(x,y)\frac{\delta^p}{\mass(B_{\dist(x,y)}(x))\dist(x,y)^p}\dd\mass(y)\\&\qquad\le \int_{\XX\setminus B_\delta(x)} \frac{\delta^p}{\mass(B_{\dist(x,y)}(x))\dist(x,y)^p}\dd\mass(y)\le C.
\end{split}
	\end{equation}
Fix $\bar x\in\XX$  and $\bar r>0$ such that the support of $f$ is contained in $\hat B_1\defeq B_{\bar r}(\bar x)$, let also $\hat B_2\defeq B_{\bar r+1}(\bar x)$. 

Now we compute, taking into account that if $x\notin \hat B_2$ and $y\notin \hat B_1$ then $f(x)=f(y)=0$, by Fubini,
\begin{equation}\label{cndcadc}
	\begin{split}
		&\int_{\XX\setminus \hat B_2}\int_\XX\chi_{\{|f(x)-f(y)|>\delta\}}(x,y)\frac{\delta^p}{\mass(B_{\dist(x,y)}(x))\dist(x,y)^p}\dd\mass(y)\dd\mass(x) \\&\qquad \le\int_{\XX\setminus \hat B_2}\int_{\hat B_1}\frac{\delta^p}{\mass(B_{\dist(x,y)}(x))\dist(x,y)^p}\dd\mass(y)\dd\mass(x)
		\\&\qquad=\delta^p\int_{\hat B_1}\int_{\XX\setminus \hat B_2} \frac{1}{\mass(B_{\dist(x,y)}(x))\dist(x,y)^p}\dd\mass(x)\dd\mass(y)
	\\&\qquad\le\delta^p\int_{\hat B_1}\int_{\XX\setminus B_1(y)} \frac{1}{\mass(B_{\dist(x,y)}(x))\dist(x,y)^p}\dd\mass(x)\dd\mass(y)\le \delta^p \mass(\hat B_1) C,
	\end{split}
\end{equation}
where the last inequality is due to \eqref{lemma28eq}.

Also, by \eqref{csdcs} we can use dominated convergence together with \eqref{blowupeq} of Lemma \ref{blowup}  or \eqref{blowupeq2} of Lemma \ref{blowuph}  to infer that
\begin{equation}\label{cndcadc1}
\begin{split}
			&\lim_{\delta\searrow 0}\int_{\hat B_2}\int_\XX\chi_{\{|f(x)-f(y)|>\delta\}}(x,y)\frac{\delta^p}{\mass(B_{\dist(x,y)}(x))\dist(x,y)^p}\dd\mass(y)\dd\mass(x)\\&\qquad=\int_{\hat B_2}\frac{C_{n,p}}{p \omega_n} \lip(f)^p(x)\dd\mass(x).
\end{split}
\end{equation}
Therefore, \eqref{mainpropeq} follows from \eqref{cndcadc} and \eqref{cndcadc1}.
\end{proof}

\begin{rem}
    It is worth noticing that Lemma  \ref{blowup} and Proposition \ref{mainprop} hold also in \say{local version}, i.e.\ integrating on an open set $\Omega\subseteq\XX$, with the same proof.  For example, if $(\XX,\dist,\mass)$ is a $\PI$ space satisfying Assumption \ref{Ass1}, $\Omega\subseteq\XX$ is open  and  $f\in\LIP_{\mathrm{bs}}(\XX)$, then, for every $p\in(1,\infty)$,
\begin{equation}\notag
		\lim_{\delta\searrow 0} \iint_{\{(x,y)\in\Omega\times\Omega:|f(x)-f(y)|>\delta\}}\frac{\delta^p}{\mass(B_{\dist(x,y)}(x))\dist(x,y)^p}\dd\mass(y)\dd\mass(x)=\frac{C_{n,p}} {p \omega_n}\int_\Omega \lip(f)^p\dd\mass,
\end{equation}
	where $C_{n,p}$ is as in Definition \ref{defnormconst}.
 \fr 
\end{rem}

\subsection{Proof of the main results}\label{sectproof}
We recall first the definition of the functionals $I_\delta$ and $J_\delta$ in Definition \ref{IandJ}. 
\begin{proof}[Proof of Theorem \ref{mainthm1}]
	 Assume first $f\in W^{1,p}(\XX)$. Then by Theorem \ref{mainthm} it holds that $$\sup_{\delta\in(0,1)}I_\delta(f)<\infty.$$
	We can then take the supremum for $\epsilon\in(0,1)$ in \eqref{IandJeq} of and prove the claim. 

	Conversely, assume that 
	$$
			\sup_{\delta\in(0,1)}\iint_{|f(x)-f(y)|\le 1}\frac{\delta|f(x)-f(y)|^{p+\delta}}{\mass(B_{\dist(x,y)}(x))\dist(x,y)^p}\dd\mass(y)\dd\mass(x)<\infty$$
			and that
				$$
			\iint_{|f(x)-f(y)|>  1}\frac{1}{\mass(B_{\dist(x,y)}(x))\dist(x,y)^p}\dd\mass(y)\dd\mass(x)<\infty.
			$$
			Therefore,
				$$\sup_{\delta\in(0,1)}\iint_{\XX\times\XX}\frac{\delta(|f(x)-f(y)|\wedge 1)^{p+\delta}}{\mass(B_{\dist(x,y)}(x))\dist(x,y)^p}\dd\mass(y)\dd\mass(x)<\infty,$$
						so that we can follow  \cite[proof of the lower bound of Theorem 1.5]{DiMarSquas} and conclude that $f\in W^{1,p}(\XX)$.
\end{proof}
\begin{proof}[Proof of Theorem \ref{thmlimits}] With 
Proposition \ref{mainprop} and Theorem \ref{cart} in mind, the proof is not so different from the one of \cite[Proposition 3.3]{NguyenPinamontiSquassinaVecchi}.
By \cite{ACM14} (and an immediate approximation argument), we have a sequence $(f_l)_l\subseteq\LIP_{\mathrm{bs}}(\XX)$ with $f_l\rightarrow f$ in $W^{1,p}(\XX)$.
 We use \eqref{trick} of Lemma \ref{tricklem} to infer that for every $l$, for every $\delta,\epsilon\in (0,1)$,
\begin{equation}\label{csdcas}
	I_\delta(f)\le  \frac{1}{\epsilon^p}I_{\epsilon\delta}(f-f_l)+\frac{1}{(1-\epsilon)^p}I_{(1-\epsilon)\delta}(f_l).
\end{equation}
Therefore, using \eqref{mainpropeq} of Proposition \ref{mainprop} and Theorem \ref{cart} (there the definition of $C_U$)
\begin{equation}
	\limsup_{\delta\searrow 0} I_{\delta}(f)\le\frac{C_U}{\epsilon^p} \Ch_p(f-f_l)+\frac{1}{(1-\epsilon)^p} \frac{C_{n,p}} {p \omega_n} \Ch_p(f_l),
\end{equation}
as the $\limsup$ is subadditive.
We now let $l\rightarrow\infty$ to infer that
\begin{equation}
	\limsup_{\delta\searrow 0} I_\delta(f)\le\frac{1}{(1-\epsilon)^p} \frac{C_{n,p}} {p \omega_n} \Ch_p(f)
\end{equation}
and then we let $\epsilon\searrow 0$ to infer
\begin{equation}\label{conc1}
	\limsup_{\delta\searrow 0} I_\delta(f)\le \frac{C_{n,p}} {p \omega_n} \Ch_p(f).
\end{equation}

Exchanging the roles of $f$ and $f_l$ in \eqref{csdcas},
\begin{equation}
	I_\delta(f_l)- \frac{1}{\epsilon^p}I_{\epsilon\delta}(f-f_l)\le\frac{1}{(1-\epsilon)^p}I_{(1-\epsilon)\delta}(f),
\end{equation}
so that, again by \eqref{mainpropeq} of Proposition  \ref{mainprop}  and Theorem \ref{cart} (there the definition of $C_U$),
\begin{equation}
	\frac{C_{n,p}} {p \omega_n} \Ch_p(f_l)-\frac{C_U}{\epsilon^p} \Ch_p(f-f_l)\le\frac{1}{(1-\epsilon)^p}\liminf_{\delta\searrow 0}I_{\delta}(f),
\end{equation}
as the $\liminf$ is superadditive. Again, we let $l\rightarrow\infty$ 
to infer that
\begin{equation}
	\frac{C_{n,p}} {p \omega_n} \Ch_p(f)\le \frac{1}{(1-\epsilon)^p}\liminf_{\delta\searrow 0}I_{\delta}(f)
\end{equation}
and then we let  $\epsilon\searrow 0$ to infer that 
\begin{equation}\label{conc2}
\frac{C_{n,p}} {p \omega_n} \Ch_p(f)\le \liminf_{\delta\searrow 0}I_{\delta}(f).
\end{equation}

Therefore, \eqref{mainthmeq} follows from \eqref{conc1} and \eqref{conc2}. Also, we can use \eqref{mainthmeq} to let $\epsilon\searrow 0$ in  \eqref{IandJeq} and deduce  \eqref{mainthmeq1}.
\end{proof}

\textbf{Acknowledgements:} We warmly thank Gioacchino Antonelli for some useful discussions on the content of \cite{Ant23} and Hoai-Minh Nguyen for helpful  comments about an earlier version of this manuscript. Part of this work was developed while the first named author was visiting the Università di Trento; the first named author wishes to thank the instutute for its kind hospitality.

\bibliographystyle{alpha}
\bibliography{biblio} 

\end{document}